\documentclass[conference]{IEEEtran}
\IEEEoverridecommandlockouts

\usepackage{amsmath}
\usepackage{amssymb}
\usepackage{amsthm}
\usepackage{graphicx}
\usepackage{cite}
\usepackage{xcolor}
\usepackage{eso-pic}

\newtheorem{theorem}{Theorem}[section]
\newtheorem{proposition}{Proposition}[section]
\theoremstyle{remark}
\newtheorem{remark}{Remark}[section]

\newcommand{\bsig}{\boldsymbol{\sigma}}
\newcommand{\bomega}{\boldsymbol{\omega}}
\newcommand{\bh}{\mathbf{h}}
\newcommand{\bu}{\mathbf{u}}
\newcommand{\bH}{\mathbf{H}}
\newcommand{\bJ}{\mathbf{J}}
\newcommand{\bv}{\mathbf{v}}
\newcommand{\bz}{\mathbf{z}}
\newcommand{\bI}{\mathbf{I}}
\newcommand{\bR}{\mathbf{R}}
\newcommand{\bD}{\mathbf{D}}
\newcommand{\bxi}{\boldsymbol{\xi}}
\newcommand{\bzero}{\mathbf{0}}
\DeclareMathOperator{\diag}{diag}
\DeclareMathOperator{\sign}{sign}

\title{Explicit Globally Asymptotically Stable Attitude Control
with Reaction Wheel Angular Momentum and Torque Constraints}

\author{
  \IEEEauthorblockN{Javier Urrios, Rafael Vazquez and Guillermo Pacheco-Ramos}

  \thanks{J. Urrios, R. Vazquez and G. Pacheco-Ramos are with the Department
  of Aerospace Engineering, Universidad de Sevilla, Camino de los
  Descubrimientos s.n., 41092 Sevilla, Spain. Emails: \texttt{furrios@us.es}, \texttt{rvazquez1@us.es}, \texttt{gpacheco@us.es}}
}

\usepackage[top=54pt, bottom=54pt, left=54pt, right=54pt]{geometry}

\begin{document}
\newgeometry{top=72pt, bottom=54pt, left=54pt, right=54pt}
\maketitle

\begin{abstract}
We present an explicit Lyapunov attitude controller for rigid spacecraft actuated by three orthogonally mounted reaction wheels, parameterized using Modified Rodrigues Parameters (MRPs), which enforces hard bounds on wheel angular momentum and applied torque at every instant without solving any online optimization.  A barrier-inspired momentum feedback prevents wheel saturation; a continuous positive scalar uniformly scales the feedback whenever a torque limit would otherwise be breached, preserving Lyapunov convergence.  For the regulation problem, global asymptotic stability (GAS) is established via LaSalle's invariance principle, and constraint satisfaction is proved analytically for arbitrary symmetric positive-definite inertia matrices.  Replacing the barrier by an error-state counterpart extends the design to tracking, with asymptotic (though no longer global) stability and the same safety guarantees under explicit feasibility conditions on the reference.  Simulations confirm convergence and constraint satisfaction, and compare performance against other controllers.
\end{abstract}

\section{Introduction}

Spacecraft attitude control with reaction wheels faces a fundamental constraint-handling challenge: achieving GAS while rigorously enforcing hard limits on wheel momentum $|h_i|\leq h_{\max}$ (beyond which saturation causes loss of control authority) and applied torque $|u_i|\leq u_{\max}$ (imposed by motor amplifiers).  Violating either constraint can result in mission loss. Classical Lyapunov designs~\cite{schaub1996stereo,tsiotras1996stabilization,tsiotras1998further,wen1991attitude} are explicit and globally stabilizing but assume unbounded actuators, delegating momentum management to auxiliary systems.  No existing approach simultaneously achieves \emph{explicit} feedback (no online optimization), \emph{GAS}, and \emph{hard constraint satisfaction} for both $h_{\max}$ and $u_{\max}$; existing solutions cluster into families sacrificing at least one property.

\smallskip
\noindent\textbf{Saturated torque control (Explicit~$+$~GAS, torque only).}  Su \& Zheng~\cite{su2011saturated,su2016unified,su2020single} proved that component-wise torque saturation preserves GAS, with extensions to disturbances~\cite{wallsgrove2005saturated} and backstepping~\cite{ali2010backstepping}; wheel momentum dynamics are ignored.

\smallskip
\noindent\textbf{CLF-CBF with QP (hard constraints, not explicit).}  These methods~\cite{ames2019cbf,shahraki2025cbf,wu2022attitude} blend a Control Lyapunov Function (CLF, for convergence) with Control Barrier Functions (CBFs, for momentum and torque bounds) in a Quadratic Program (QP) solved at every cycle: the constraints are handled rigorously, but the CLF row must be relaxed to keep the QP feasible, so convergence is promoted rather than guaranteed, and the online-optimization cost is unsuitable for resource-constrained platforms.

\smallskip
\noindent\textbf{BLFs (Explicit~$+$~hard constraints, not GAS).}  Barrier Lyapunov Functions (BLFs)~\cite{tee2009barrier,zhang2024finite} create repulsive potentials near state constraint boundaries, yielding explicit laws, but stability is only semi-global and, to the authors' knowledge, no BLF design on reaction wheel constraints has been shown to be GAS. This lack of global stability is primarily attributed to the vanishing of the descent property of the Lyapunov function in the vicinity of the constraints.

\smallskip
\noindent\textbf{ERGs (Explicit~$+$~hard constraints, conditional).}  Explicit Reference Governors (ERGs)~\cite{nicotra2020explicit,nakano2023explicit,dang2022explicit} pre-compute admissible sets offline for fast online lookups, but the 9 spacecraft-plus-wheels states cause exponential complexity and convergence is guaranteed only to the nearest admissible reference.

\smallskip
This paper presents an explicit attitude control law enforcing hard momentum and torque constraints at every instant, with GAS guaranteed for the regulation problem.  The design combines a barrier-inspired momentum feedback (a diagonal matrix $\bD$ amplifying damping near saturation) with a scalar $s$ preventing torque violations; constraint satisfaction is proved analytically for arbitrary inertia matrices.  The controller requires $\mathcal{O}(1)$ operations per cycle (a $200$x speedup over CLF-CBF-QP baselines) and offers no solver failure modes, deterministic execution time, and straightforward flight-software certification.  An extension to tracking establishes asymptotic (though not global) stability under explicit, a-priori checkable conditions, and simulations validate convergence and safety under near-saturated conditions.

A GAS result is attainable because the constraints act on the actuator states and forbid no attitude; the precise meaning of GAS is discussed after Theorem~\ref{thm:stability}.

The paper is organized as follows: Section~\ref{sec:dynamics} introduces the dynamics, Section~\ref{sec:ctrl} the regulation law and its stability and safety proofs, Section~\ref{sec:tracking} the tracking extension, Section~\ref{sec:sim} simulations, and Section~\ref{sec:conclusion} conclusions.

\clearpage
\restoregeometry

\section{Attitude Dynamics}
\label{sec:dynamics}

\subsection{Problem Statement}

Consider a rigid body with three identical orthogonally mounted reaction wheels, not necessarily aligned with the principal inertia axes. Each wheel has a maximum angular momentum $h_{\max}$, and a maximum torque $u_{\max}$. The control objective is to stabilize the attitude to a desired equilibrium while respecting the maximum angular momentum and torque limits at all times. No disturbances or model uncertainties are considered, and it is assumed that the full state (attitude and angular velocity) and the inertia are available for feedback. The body frame is defined such that each reaction wheel spins about each body axis, $\mathbf{e}_1$, $\mathbf{e}_2$, and $\mathbf{e}_3$.

\subsection{Attitude Representation}

Attitude is represented by Modified Rodrigues Parameters (MRPs) $\bsig\in\mathbb{R}^3$~\cite{shuster1993survey}, limited to $\|\bsig\|^2\leq 1$ to avoid the kinematic singularity. Whenever $\|\bsig\|^2=1$ the shadow-set switch $\bsig\leftarrow-\bsig/\|\bsig\|^2$ is applied.  Kinematics read
\begin{equation}
  \dot{\bsig} = C(\bsig)\,\bomega, \quad
  C(\bsig) = \tfrac{1}{4}\!\left[
    (1-\|\bsig\|^2)\bI_n + 2\bsig^{\times} + 2\bsig\bsig^T
  \right],
  \label{eq:kinematics}
\end{equation}
with $\bI_n$ the $n\times n$ identity matrix, $\bsig^{\times}$ the skew-symmetric matrix of $\bsig$, and $\bomega$ the angular velocity in the body frame.

\subsection{Dynamics and Conservation Law}

Since the wheels are aligned with the body axes, each component $h_i$ of their total angular momentum vector $\bh=(h_1,h_2,h_3)^T$ is the angular momentum of the $i$-th wheel.  The Euler--Poinsot equations with wheel actuation are
\begin{align}
  \bJ\dot{\bomega} + \bomega\times(\bJ\bomega+\bh) &= -\bu,
  \label{eq:euler} \\
  \dot{\bh} &= \bu,
  \label{eq:wheels}
\end{align}
where $\bJ=\bJ^T\succ 0$ is the body inertia matrix and $\bu$ the control torque vector. System~\eqref{eq:euler}--\eqref{eq:wheels} conserves the total angular momentum in inertial coordinates:
\begin{equation}
  \bJ\bomega + \bh = \bH = \bR(\bsig)\,\bH_0,
  \label{eq:conservation}
\end{equation}
where $\bR(\bsig)$ is the rotation matrix (body to inertial) and $\bH_0$ is the initial inertial angular momentum with $H_0\triangleq\|\bH_0\|$.

\smallskip
\noindent\textbf{Assumption~(A1)}: $H_0 < h_{\max}$.

This ensures $|H_i|\leq H_0 < h_{\max}$, so any attitude can be held stationary with all wheels within the momentum bound.  Note also that, although the plant \eqref{eq:kinematics}--\eqref{eq:wheels} evolves in a 9D space, \eqref{eq:conservation} slaves $\bh=\bR(\bsig)\bH_0-\bJ\bomega$ to $(\bsig,\bomega)$ along every trajectory: the closed-loop state is $(\bsig,\bomega)\in\mathbb{R}^3\times\mathbb{R}^3$, and all stability claims below refer to this six-dimensional space.

\section{Feedback Control Law}
\label{sec:ctrl}

\subsection{Lyapunov Function and Regulation Control Law}

Consider the Lyapunov function
\begin{equation} \label{eq:lyap_reg}
  V \triangleq 2k_p\log(1+\|\bsig\|^2) + \tfrac{1}{2}\bomega^T\bJ\bomega,
\end{equation}
with
\begin{equation} \label{eq:kp_bound}
  k_p \in \!\left(0,\;\frac{u_{\max}}{\sqrt{1+R_{\max}}}\right),
  \quad R_{\max} = \max_{i\neq j}\!\left(\frac{J_{ij}}{J_{ii}}\right)^{\!2}.
\end{equation}
Here $k_p$ acts as a proportional (attitude) gain and $k_d$, introduced below, as a damping gain; the bound in \eqref{eq:kp_bound} reserves torque authority for the damping term to dominate near momentum saturation despite inertia cross-couplings, as the proof of Theorem~\ref{thm:safety} makes precise.  Using the MRP identity $\bsig^T C(\bsig)=\tfrac{1}{4}(1+\|\bsig\|^2)\bsig^T$ and $\bomega^T(\bomega\times\cdot)=0$, the time derivative of $V$ along \eqref{eq:kinematics}--\eqref{eq:euler} is
\begin{equation}
  \dot{V} = \bomega^T(k_p\bsig - \bu).
  \label{eq:Vdot}
\end{equation}
Define the positive-diagonal matrix encoding the proximity of each wheel to its momentum limit,
\begin{equation}
  \bD \triangleq \diag\!\left(
    \tfrac{1}{|h_{\max}^2-h_1^2|},\;
    \tfrac{1}{|h_{\max}^2-h_2^2|},\;
    \tfrac{1}{|h_{\max}^2-h_3^2|}
  \right),
  \label{eq:D}
\end{equation}
the momentum-feedback vector
\begin{equation}
  \bz \triangleq -k_d\,\bJ\bD\bJ\,\bomega, \quad k_d>0,
  \label{eq:w}
\end{equation}
and, for each $i\in\{1,2,3\}$ with $z_i\neq 0$,
\begin{equation}
  s_i \triangleq \frac{u_{\max} + \sign(z_i)\,k_p\sigma_i}{|z_i|},
  \label{eq:si}
\end{equation}
with $s_i\triangleq+\infty$ when $z_i=0$. Set $s\triangleq\min(1,s_1,s_2,s_3)$.  The control law is
\begin{equation}
  \bu = k_p\bsig - s\,\bz.
  \label{eq:ctrl}
\end{equation}

\begin{remark}[Well-posedness and role of $s$]
\label{rem:s}
Since $\|\bsig\|\leq 1$, the numerator of \eqref{eq:si} satisfies $u_{\max}+\sign(z_i)k_p\sigma_i\geq u_{\max}-k_p>0$ by \eqref{eq:kp_bound}, so $s_i>0$ and $s\in(0,1]$.  The growth of $s_i$ as $z_i\to 0$ causes no singularity: whenever $|z_i|<u_{\max}-k_p$, $s_i>1$ drops out of the minimum ($s=1$ when $\bz=\bzero$), so $s$ is locally a minimum of finitely many locally Lipschitz functions, $\bu$ is locally Lipschitz wherever $|h_i|\neq h_{\max}$ for all $i$, and closed-loop solutions are unique.  Moreover, $s$ scales only the damping term $-\bz$: any $s\in(0,1]$ preserves $\dot V\leq 0$ in \eqref{eq:Vdot2} below, affecting only the convergence rate.
\end{remark}

\subsection{Stability}

\begin{theorem}[Global asymptotic stability]
  \label{thm:stability}
  Consider model \eqref{eq:kinematics}--\eqref{eq:wheels} under control~\eqref{eq:ctrl} with coefficients $k_d>0$ and $k_p$ verifying \eqref{eq:kp_bound}. Under Assumption~(A1), the equilibrium
  $(\bsig,\bomega)=(\bzero,\bzero)$ is GAS.
\end{theorem}
\begin{proof}
  Substituting \eqref{eq:ctrl} into \eqref{eq:Vdot},
  \begin{equation}
    \dot{V} = s\,\bomega^T\bz =
    -s\,k_d\,\bomega^T\bJ\bD\bJ\,\bomega \leq 0,
    \label{eq:Vdot2}
  \end{equation}
  since $s>0$, $k_d>0$, $\bJ\bD\bJ\succ 0$; equality holds if and only if $\bomega=\bzero$.  Hence $\dot V$ vanishes exactly on $E\triangleq\{(\bsig,\bomega):\bomega=\bzero\}$.  Let $\mathcal{M}$ be the largest invariant set contained in $E$.  Any solution remaining in $\mathcal{M}$ satisfies $\bomega\equiv\bzero$, hence $\dot\bomega\equiv\bzero$ and, by \eqref{eq:w}, $\bz\equiv\bzero$; then \eqref{eq:euler} and \eqref{eq:ctrl} give $\bzero=\bu=k_p\bsig$, so $\bsig=\bzero$.  Thus $\mathcal{M}=\{(\bzero,\bzero)\}$ and, since $V$ is radially unbounded (its sublevel sets are compact and positively invariant), LaSalle's invariance principle guarantees that every trajectory converges to $(\bzero,\bzero)$; together with Lyapunov stability (from $\dot V\leq 0$ and positive definiteness of $V$), this gives GAS.
\end{proof}

At the shadow-set switch $\bsig\leftarrow-\bsig/\|\bsig\|^2$ (applied at $\|\bsig\|^2=1$), $V$ is continuous and $\dot{V}=s\,\bomega^T\bz$ does not involve $\bsig$, so the Lyapunov argument is unaffected; the control, however, is discontinuous there, which may cause chattering. Some discontinuity is unavoidable: no continuous time-invariant feedback globally stabilizes the rigid-body attitude~\cite{bhat2000topological}, so the GAS claim of Theorem~\ref{thm:stability} is meant in MRP coordinates under this (discontinuous) shadow-set switching. Quaternions avoid this at the cost of unwinding, whereas MRPs are minimal and confine the discontinuity to $\|\bsig\|=1$, the $180^\circ$ rotations farthest from the target.

\subsection{Safety}

\begin{theorem}[Constraint satisfaction]
  \label{thm:safety}
  Consider model \eqref{eq:kinematics}--\eqref{eq:wheels} under control~\eqref{eq:ctrl} with coefficients $k_d>0$ and $k_p$ verifying \eqref{eq:kp_bound}. Under Assumption~(A1), it holds that $|u_i(t)|\leq u_{\max}$ and $|h_i(t)|<h_{\max}$ $\forall t\geq 0$, $i\in\{1,2,3\}$.
\end{theorem}
\begin{proof}
  \textit{Torque.} If $z_i=0$, then
  $|u_i|=k_p|\sigma_i|\leq k_p<u_{\max}$ by \eqref{eq:kp_bound}.  If $z_i\neq 0$, then $0<s\leq s_i$ and \eqref{eq:si} give $0<s|z_i|\leq u_{\max}+\sign(z_i)k_p\sigma_i$, so $u_i=k_p\sigma_i-s z_i$ lies between $k_p\sigma_i$ and $k_p\sigma_i-\sign(z_i)\left(u_{\max}+\sign(z_i)k_p\sigma_i\right) =-\sign(z_i)u_{\max}$, both between $\pm u_{\max}$, so $|u_i|\leq u_{\max}$.

  \textit{Angular momentum.}
  We prove $h_1(t)<h_{\max}$ for all $t\geq 0$. Arguments for $h_1>-h_{\max}$ and $i\neq 1$ are symmetric.  Suppose, for contradiction, that a wheel momentum reaches its bound in finite time; let $T$ be the first such time and let $h_1$ alone attain it, so that near $T$ the entries $D_{22},D_{33}$ remain bounded, as does $\bomega$ (since $\dot V\leq 0$).  Since $|H_1|\leq H_0<h_{\max}$, by continuity $h_1>\tfrac{1}{2}(h_{\max}+H_0)$ for $t$ close to $T$, and~\eqref{eq:conservation} gives $(\bJ\bomega)_1 = H_1-h_1 < -\tfrac{1}{2}(h_{\max}-H_0)<0$. As $h_1\to h_{\max}$, $D_{11}=|h_{\max}^2-h_1^2|^{-1}\to+\infty$, so expanding \eqref{eq:w} componentwise,
  \begin{equation}
    z_i = -k_d J_{i1}D_{11}(\bJ\bomega)_1 + \rho_i,\quad
    \rho_i \triangleq -k_d\!\!\sum_{k=2,3}\!J_{ik}D_{kk}(\bJ\bomega)_k,
    \label{eq:z_expand}
  \end{equation}
  with every $\rho_i$ bounded: thus $z_1\to+\infty$ and $z_i/z_1\to J_{i1}/J_{11}$ for $i\neq 1$. We now use this to show that, when $h_1\to h_{\max}$, $u_1=\dot h_1<0$.  At each instant, one of three cases holds.

  \emph{Case~1} ($s=1$): $u_1 = k_p\sigma_1 - z_1\to-\infty$, since $k_p|\sigma_1|\leq k_p$ while $z_1\to+\infty$.

  \emph{Case~2} ($s=s_1<1$): since $z_1>0$ near $T$, $u_1 = k_p\sigma_1 - \bigl(u_{\max}+\sign(z_1)k_p\sigma_1\bigr)z_1/|z_1| = -u_{\max}$: the wheel is driven away from the bound at the maximal rate.

  \emph{Case~3} ($s=s_2<1$, taking $j=2$ as binding axis without loss of generality): note first that $J_{12}\neq 0$ necessarily, for otherwise \eqref{eq:z_expand} leaves $z_2=\rho_2$ bounded and $s_2$ bounded below, while $s_1\to 0$, so $s_2$ could not be the minimum near $T$.  By the limits above, $z_1/|z_2|\to r\triangleq J_{11}/|J_{12}|$ and $\sign(z_2)\to\sign(J_{12})$, so
  \begin{align}
    u_1 &= k_p\sigma_1
           - \frac{u_{\max}+\sign(z_2)k_p\sigma_2}{|z_2|}\,z_1 \notag\\
        &= k_p\!\left(\sigma_1-\sign(J_{12})\,r\sigma_2\right)
           - r\,u_{\max} + o(1),
    \label{eq:u1case3}
  \end{align}
  with $o(1)\to 0$ as $t\to T$.  The Cauchy--Schwarz inequality and $\|\bsig\|\leq 1$ give $|k_p(\sigma_1-\sign(J_{12})\,r\sigma_2)|\leq k_p\sqrt{1+r^2}$. Moreover, $|J_{12}|\leq J_{11}\sqrt{R_{\max}}$ by \eqref{eq:kp_bound}, so $r\geq 1/\sqrt{R_{\max}}$, $1+r^2\leq r^2(1+R_{\max})$, and $k_p\sqrt{1+r^2}\leq k_p r\sqrt{1+R_{\max}}$; hence $u_1\leq -r\left(u_{\max}-k_p\sqrt{1+R_{\max}}\right)+o(1)<0$ by~\eqref{eq:kp_bound}.  For a different binding axis $j$, replace $J_{12}$ by $J_{1j}$ and $r$ by $r_j=J_{11}/|J_{1j}|\geq 1/\sqrt{R_{\max}}$.

  In all three cases $u_1$ is negative and bounded away from zero as $t\to T$, no matter how the cases alternate; integrating $\dot h_1=u_1$ over $[T-\delta,T]$ with $\delta$ small enough then gives $h_1(T)<h_1(T-\delta)<h_{\max}$, contradicting $h_1(T)=h_{\max}$.  Hence $h_1(t)<h_{\max}$ for all $t\geq 0$.
\end{proof}

\begin{remark}[Diagonal inertia simplification]
  When $\bJ=\diag(J_1,J_2,J_3)$, the matrix $\bJ\bD\bJ$ is diagonal with entries $J_i^2/|h_{\max}^2-h_i^2|$, so the momentum feedback simplifies to $z_i = -k_d J_i^2\omega_i/|h_{\max}^2-h_i^2|$. Furthermore, since the axes decouple, $k_d$ and $s$ may be replaced by per-axis gains $k_{d,i}>0$ and saturations $s_i$, and the bound on $k_p$ relaxes to $k_p<u_{\max}$.
\end{remark}

\begin{remark}[Initial conditions outside the safe set]
  Theorems~\ref{thm:stability} and~\ref{thm:safety} remain valid even if $|h_i(0)|>h_{\max}$ for some $i$: the absolute value in \eqref{eq:D} keeps the corresponding entry of $\bD$ positive, so $\bz$ still drives the state towards the equilibrium, eventually pushing $h_i$ back within the safe bounds.
\end{remark}

\section{Extension to Attitude Tracking}
\label{sec:tracking}

We now extend the controller to track a desired trajectory $(\bsig_d(t),\bomega_d(t))$ while still enforcing both hard constraints, by designing the same type of law in error coordinates.

\subsection{Error Coordinates and Dynamics}

Let $\bsig_d(t)$, $\bomega_d(t)$ be the desired attitude and angular velocity.  The error MRP $\bsig_e$ represents the rotation of the current frame relative to the desired frame, with $\bomega_e = \bomega - \bR(\bsig_e)\bomega_d$ the body-frame velocity error.  Error kinematics take the standard form $\dot{\bsig}_e = C(\bsig_e)\bomega_e$. Differentiating $\bomega_e$ and using~\eqref{eq:euler}:
\begin{equation}
  \bJ\dot{\bomega}_e = \underbrace{-\bomega\!\times\!(\bJ\bomega+\bh)
  - \bJ\!\left(\dot\bR(\bsig_e)\bomega_d
  + \bR(\bsig_e)\dot\bomega_d\right)}_{\displaystyle\bxi}
  - \bu,
  \label{eq:err_dyn}
\end{equation}
where $\bxi$ collects all computable feedforward terms.

\subsection{Lyapunov Function and Tracking Control Law}

We use, in error coordinates, the same Lyapunov function as in \eqref{eq:lyap_reg}, $V_e \triangleq 2k_p\log(1+\|\bsig_e\|^2) + \tfrac{1}{2}\bomega_e^T\bJ\bomega_e$, where \eqref{eq:kp_bound} still holds for $k_p$. Computing its time derivative along \eqref{eq:err_dyn} exactly as in \eqref{eq:Vdot} (with the MRP identity applied to $\dot\bsig_e=C(\bsig_e)\bomega_e$):
\begin{equation}
  \dot{V}_e = \bomega_e^T(k_p\bsig_e + \bxi - \bu).
  \label{eq:Vdot_track}
\end{equation}

Note the extra term $\bxi$: in regulation the gyroscopic part cancels in $\dot V$ via $\bomega^T(\bomega\times\cdot)=0$, but here $\bomega_e\neq\bomega$ in general, so it must be cancelled explicitly together with the reference feedforward.

Define the desired wheel momentum at tracking equilibrium ($\bomega_e=\bzero$) as $\bh_d \triangleq \bH - \bJ\bR(\bsig_e)\bomega_d$. From~\eqref{eq:conservation}:
\begin{equation}
  \bh - \bh_d = \bH - \bJ\bomega - \bh_d = -\bJ\bomega_e,
  \label{eq:h_err}
\end{equation}
which vanishes at $\bomega_e=\bzero$.  Define also $\bD$ as in~\eqref{eq:D} (using the actual $\bh$), the tracking momentum feedback
\begin{equation}
  \bz_e \triangleq -k_d\,\bJ\bD\bJ\,\bomega_e, \quad k_d>0,
  \label{eq:we}
\end{equation}
and, for each $i\in\{1,2,3\}$ (with $s_{e,i}\triangleq+\infty$ when $z_{e,i}=0$, as in Remark~\ref{rem:s}),
\begin{equation}
  s_{e,i} \triangleq
  \frac{u_{\max} + \sign(z_{e,i})(k_p\sigma_{e,i}+\xi_i)}{|z_{e,i}|},
  \;\;
  s_e \triangleq \min_i\,(1,\,s_{e,i}).
  \label{eq:se}
\end{equation}
The proposed control law for tracking is therefore
\begin{equation}
  \bu = k_p\bsig_e + \bxi - s_e\,\bz_e.
  \label{eq:ctrl_track}
\end{equation}
Define the parameter-dependent threshold
\begin{equation}
  \xi^* \triangleq \frac{u_{\max} - k_p\sqrt{1+R_{\max}}}{1+\sqrt{R_{\max}}}.
  \label{eq:xi_star}
\end{equation}
Condition~\eqref{eq:kp_bound} guarantees $\xi^*>0$ and $k_p + \xi^*\leq u_{\max}$ (equality only when $R_{\max}=0$).

The following explicit, physically interpretable conditions are sufficient (not necessary, hence conservative) for stability and safety; $\lambda_{\max}(\cdot)$ and $\lambda_{\min}(\cdot)$ denote extreme eigenvalues.

\begin{proposition}[Tracking stability and safety]
  \label{prop:tracking}
  Consider
  \begin{enumerate}
    \item[(A2)] $\displaystyle\sup_{t\geq 0}\|\bomega_d(t)\|
      < \dfrac{h_{\max}-H_0}{\lambda_{\max}(\bJ)}$,
    \item[(A3)] $\|\bxi(t)\|_\infty < \xi^*$ for all $t\geq 0$,
  \end{enumerate}
  and model~\eqref{eq:kinematics}--\eqref{eq:wheels} under control~\eqref{eq:ctrl_track} with coefficients $k_d>0$ and $k_p$ verifying \eqref{eq:kp_bound}. Then, under assumptions (A1), (A2) and (A3), the system is asymptotically stable at $(\bsig_e,\bomega_e)=(\bzero,\bzero)$, and satisfies $|u_i(t)|\leq u_{\max}$, $|h_i(t)|<h_{\max}$ $\forall t\geq 0$.
\end{proposition}

\begin{proof}
  \textit{Step~1 (positivity of $s_e$).}
  From (A3) and $k_p|\sigma_{e,i}|\leq k_p$: $|k_p\sigma_{e,i}+\xi_i|\leq k_p+|\xi_i| < k_p+\xi^*\leq u_{\max}$, where the last inequality comes from~\eqref{eq:xi_star}. Hence $s_{e,i}>0$, so $s_e\in(0,1]$.

  \textit{Step~2 (asymptotic stability).}
  Substituting control~\eqref{eq:ctrl_track} into~\eqref{eq:Vdot_track}, the terms $k_p\bsig_e$ and $\bxi$ cancel and only the scaled damping term survives: $\dot{V}_e = \bomega_e^T(k_p\bsig_e+\bxi-k_p\bsig_e-\bxi+s_e\bz_e) = s_e\,\bomega_e^T\bz_e$. Since $\bJ\bD\bJ\succ 0$, $s_e>0$, $k_d>0$, by \eqref{eq:we}:
  \begin{equation}
    \dot{V}_e = -s_e\,k_d\,\bomega_e^T\bJ\bD\bJ\,\bomega_e \leq 0,
    \label{eq:Vdot_track2}
  \end{equation}
  with equality for $\bomega_e=\bzero$. $V_e$ is radially unbounded, so all trajectories are bounded. Let $\mathcal{M}=\{(\bsig_e,\bomega_e):\bomega_e=\bzero\}$. On $\mathcal{M}$, $\bz_e=\bzero$ and~\eqref{eq:err_dyn} with $\dot\bomega_e=\bzero$ gives $\bzero=\bxi-\bu=-k_p\bsig_e$, so $\bsig_e=\bzero$.  The largest invariant subset of $\{\dot{V}_e=0\}$ is therefore $(\bzero,\bzero)$, and by LaSalle's invariance principle every trajectory converges to $(\bsig_e,\bomega_e)=(\bzero,\bzero)$.  Asymptotic stability follows from $\dot{V}_e\leq 0$ and positive definiteness of $V_e$.

  \textit{Step~3 (torque constraints).}
  If $z_{e,i}=0$, then $|u_i|=|k_p\sigma_{e,i}+\xi_i|<u_{\max}$ by Step~1.  If $z_{e,i}\neq 0$, then $0<s_e|z_{e,i}|\leq s_{e,i}|z_{e,i}|=u_{\max}+\sign(z_{e,i})(k_p\sigma_{e,i}+\xi_i)$, so, exactly as in the torque part of Theorem~\ref{thm:safety} with $k_p\sigma_i$ replaced by $k_p\sigma_{e,i}+\xi_i$, $u_i$ lies between $k_p\sigma_{e,i}+\xi_i$ and $-\sign(z_{e,i})u_{\max}$, hence $|u_i|\leq u_{\max}$.

  \textit{Step~4 (angular momentum constraints).}
  We prove $h_1(t)<h_{\max}$ $\forall t\geq 0$; the arguments for $h_1>-h_{\max}$ and $i\neq 1$ are symmetric. Proof similar to that of Theorem~\ref{thm:safety}.

  First we bound the desired wheel momentum $\bh_d=\bH-\bJ\bR(\bsig_e)\bomega_d$.  Since $\|\bH\|=H_0$ by \eqref{eq:conservation}, $\|\bR(\bsig_e)\bomega_d\|=\|\bomega_d\|$, and $\|\bJ\bv\|\leq\lambda_{\max}(\bJ)\|\bv\|$, the triangle inequality gives
  \begin{align}
    |h_{d,1}|&\leq\|\bH - \bJ\bR(\bsig_e)\bomega_d\|
    \leq H_0+\lambda_{\max}(\bJ)\|\bomega_d\| \notag\\
    &\leq \bar h_d \triangleq
    H_0 + \lambda_{\max}(\bJ)\textstyle\sup_{t\geq0}\|\bomega_d\|
    \;<\; h_{\max},
    \label{eq:hd_bound}
  \end{align}
  where the final strict inequality is precisely (A2).

  Suppose now, for contradiction, that a wheel momentum reaches its bound in finite time; let $T$ be the first such time and let, as in Theorem~\ref{thm:safety}, $h_1$ alone attain the bound at $T$, so that $D_{22},D_{33}$ are bounded on $[T-\delta,T]$.  Recall that $\dot V_e\leq 0$ bounds $\bomega_e$.  By \eqref{eq:hd_bound} and continuity, shrinking $\delta$ if necessary, for all $t\in[T-\delta,T)$, from~\eqref{eq:h_err}, $(\bJ\bomega_e)_1=h_{d,1}-h_1\leq \bar h_d-h_1 \leq -c_e\triangleq -\tfrac{1}{2}(h_{\max}-\bar h_d)<0$. Repeating \eqref{eq:z_expand} and the subsequent limits with $\bomega_e$ in place of $\bomega$ yields $z_{e,1}\to+\infty$ and $z_{e,i}/z_{e,1}\to J_{i1}/J_{11}$ as $t\to T$.  As in Theorem~\ref{thm:safety}, it suffices to show that $u_1(t)\leq-\epsilon<0$ for all $t$ close enough to $T$; the same integral argument then gives the contradiction.  Again, at each instant one of three cases holds.

  \emph{Case~1} ($s_e=1$):
  $u_1 = (k_p\sigma_{e,1}+\xi_1) - z_{e,1}\to -\infty$, since $|k_p\sigma_{e,1}+\xi_1|<u_{\max}$ by Step~1 while $z_{e,1}\to+\infty$.

  \emph{Case~2} ($s_e=s_{e,1}<1$): since $z_{e,1}>0$ near $T$, the numerator of $s_{e,1}$ cancels the driving term ($\xi_1$ included) exactly as in Case~2 of Theorem~\ref{thm:safety}, giving $u_1=-u_{\max}<0$.

  \emph{Case~3} ($s_e=s_{e,j}<1$ for some $j\neq 1$): as in
  Theorem~\ref{thm:safety}, $J_{1j}\neq 0$, and $z_{e,1}/|z_{e,j}|\to r_j\triangleq J_{11}/|J_{1j}|\geq 1/\sqrt{R_{\max}}$ and $\sign(z_{e,j})\to\sign(J_{1j})$ as $t\to T$. Hence
  \begin{equation}
    u_1 = (k_p\sigma_{e,1}+\xi_1) - s_{e,j}\,z_{e,1}
        = \Delta_j - r_j\,u_{\max} + o(1),
    \label{eq:u1_case3_track}
  \end{equation}
  with $\Delta_j\triangleq(k_p\sigma_{e,1}+\xi_1) -\sign(J_{1j})\,r_j(k_p\sigma_{e,j}+\xi_j)$ and $o(1)\to 0$ as $t\to T$.  The Cauchy--Schwarz and triangle inequalities, with $\|\bsig_e\|\leq 1$, give
  \begin{equation}
    |\Delta_j| \leq k_p\sqrt{1+r_j^2}
              + \|\bxi\|_\infty(1+r_j).
    \label{eq:Delta_bound}
  \end{equation}
  We claim that
  \begin{equation}
    \xi^*(1+r)\leq r\,u_{\max}-k_p\sqrt{1+r^2}
    \quad\text{for all } r\geq 1/\sqrt{R_{\max}}.
    \label{eq:xistar_claim}
  \end{equation}
  To see this, consider $g(r)=(r\,u_{\max}-k_p\sqrt{1+r^2})/(1+r)$.  Its derivative is $g'(r)=\bigl[u_{\max}+k_p(1-r)/\sqrt{1+r^2}\bigr]/(1+r)^2$, whose numerator satisfies $u_{\max}+k_p(1-r)/\sqrt{1+r^2}\geq u_{\max}-k_p r/\sqrt{1+r^2} >u_{\max}-k_p>0$ (using $k_p<u_{\max}$), so $g$ is strictly increasing.  Direct substitution of \eqref{eq:xi_star} shows $g(1/\sqrt{R_{\max}})=\xi^*$, which proves the claim~\eqref{eq:xistar_claim}.  By (A3) and continuity, there exists $\mu>0$ such that $\|\bxi(t)\|_\infty\leq\xi^*-\mu$ on the compact interval $[T-\delta,T]$; combining \eqref{eq:Delta_bound} and \eqref{eq:xistar_claim},
  \begin{equation}
    |\Delta_j| \leq k_p\sqrt{1+r_j^2}+(\xi^*-\mu)(1+r_j)
             \leq r_j\,u_{\max}-\mu(1+r_j),
    \label{eq:Delta_final}
  \end{equation}
  hence $u_1\leq-\mu(1+r_j)+o(1)<0$ for $t$ close enough to $T$.

  As in Theorem~\ref{thm:safety}, the finitely many cases each bound $u_1$ negatively, uniformly in $t$ near $T$, which completes the contradiction.  Hence $h_1(t)<h_{\max}$ for all $t\geq 0$.
\end{proof}

Proposition~\ref{prop:tracking} is \emph{conditional} on (A2)--(A3), and thus not GAS: $\bxi$ contains the gyroscopic term $-\bomega\times(\bJ\bomega+\bh)$, which grows with $\|\bomega_e\|$, so a large initial error or an aggressive reference can push $\|\bxi\|_\infty$ beyond $\xi^*$, violating (A3).

(A3) is checkable a priori, as the two terms of $\bxi$ in \eqref{eq:err_dyn} admit explicit bounds.  First, since $\dot{V}_e\leq 0$, the definition of $V_e$ gives $\tfrac{1}{2}\lambda_{\min}(\bJ)\|\bomega_e(t)\|^2\leq V_e(t)\leq V_e(0)$, hence $\|\bomega_e(t)\|\leq\bar\omega_e\triangleq\sqrt{2V_e(0)/\lambda_{\min}(\bJ)}$ and $\|\bomega\|\leq\|\bomega_e\|+\|\bR(\bsig_e)\bomega_d\|\leq \bar\omega_e+\|\bomega_d\|$ for all $t\geq 0$.  By the conservation law~\eqref{eq:conservation} the gyroscopic term is simply $-\bomega\times(\bJ\bomega+\bh)=-\bomega\times\bH$ with $\|\bH\|=H_0$, so its norm is at most $(\bar\omega_e+\|\bomega_d\|)H_0$.  Second, using $\dot\bR(\bsig_e)=-\bomega_e^{\times}\bR(\bsig_e)$, $\|\dot\bR(\bsig_e)\bomega_d\|=\|\bomega_e\times\bR(\bsig_e)\bomega_d\| \leq\bar\omega_e\|\bomega_d\|$, so the reference feedforward term of \eqref{eq:err_dyn} is bounded by $\lambda_{\max}(\bJ)\bigl(\|\dot\bomega_d\|+\bar\omega_e\|\bomega_d\|\bigr)$. Since $\|\cdot\|_\infty\leq\|\cdot\|$, combining both bounds yields
\begin{align}
  \|\bxi(t)\|_\infty \leq{}& \left(\bar\omega_e+\|\bomega_d(t)\|\right)H_0
  \notag\\
  &+\lambda_{\max}(\bJ)\left(\|\dot\bomega_d(t)\|
  +\bar\omega_e\|\bomega_d(t)\|\right);
  \label{eq:xi_apriori}
\end{align}
verifying a-priori that this bound stays below $\xi^*$ (true for slow references and small initial errors) certifies (A3).

In contrast, the regulation controller (Theorems~\ref{thm:stability} and~\ref{thm:safety}) is truly GAS: with $\bomega_d=\bzero$, $\bxi$ reduces to $-\bomega\times(\bJ\bomega+\bh)$, orthogonal to $\bomega$, vanishing identically in $\dot{V}$ and imposing no restriction on the initial condition.

\section{Simulations}
\label{sec:sim}

The proposed controller is now simulated and compared against a classical controller and the OD-CLF-CBF-QP~\cite{shahraki2025cbf}.

\subsection{Setup}

Simulations use the controller~\eqref{eq:D}--\eqref{eq:ctrl}. The inertia matrix is $\bJ = R^T\diag(15,\,10,\,8)\,R$~kg\,m$^2$, $R$ being a rotation by $30^{\circ}$ about $(1,1,1)^T/\!\sqrt{3}$, giving a non-diagonal $\bJ$. Limits are $h_{\max}=5$~N\,m\,s, $u_{\max}=0.3$~N\,m; unless otherwise noted, $k_p = 0.95\,u_{\max}/\sqrt{1+R_{\max}}$, $k_d=0.3$.  Integration uses MATLAB's \texttt{ode45} with event-triggered shadow-set switching.  The classical MRP controller $\bu_{\text{cl}} = k_p\bsig + k_d\bJ\bomega$~\cite{schaub2018analytical} is used as a baseline, with identical gains.

\subsection{Regulation: Safety Under Near-Saturated Conditions}

\begin{figure}[!t]
  \centering
  \IfFileExists{figures/sim_safety.pdf}%
    {\includegraphics[width=\columnwidth]{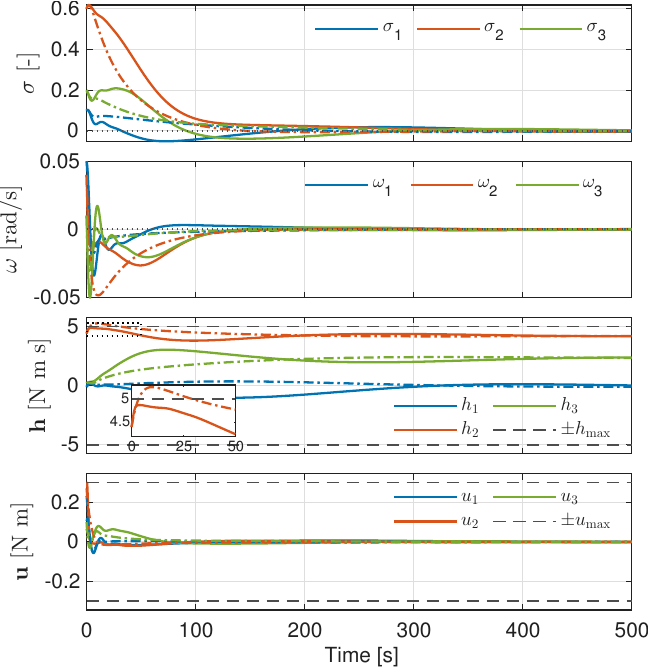}}%
    {\figplaceholder{sim\_safety: $h$ vs.\ time, proposed vs.\ classical, dashed $\pm h_{\max}$}}
  \caption{Safety stress test with $H_0 = 0.96\,h_{\max}$ ($\bsig_0 = (0.1,0.6, 0.2)^T$, $\bomega_0 = (0.05, 0.04, 0.01)^T$~rad/s, $\bh_0=(-0.1, 4.4, 0.2)^T$~N\,m\,s): proposed controller (solid) vs.\ classical MRP controller~\cite{schaub2018analytical} (dash-dot). Top to bottom: MRPs, angular velocity, wheel momenta (dashed $\pm h_{\max}$; inset: constraint violation of the classical law), wheel torque (dashed $\pm u_{\max}$).}
  \label{fig:safety}
\end{figure}

\begin{figure}[!t]
  \centering
  \IfFileExists{figures/sim_tracking.pdf}%
    {\includegraphics[width=\columnwidth]{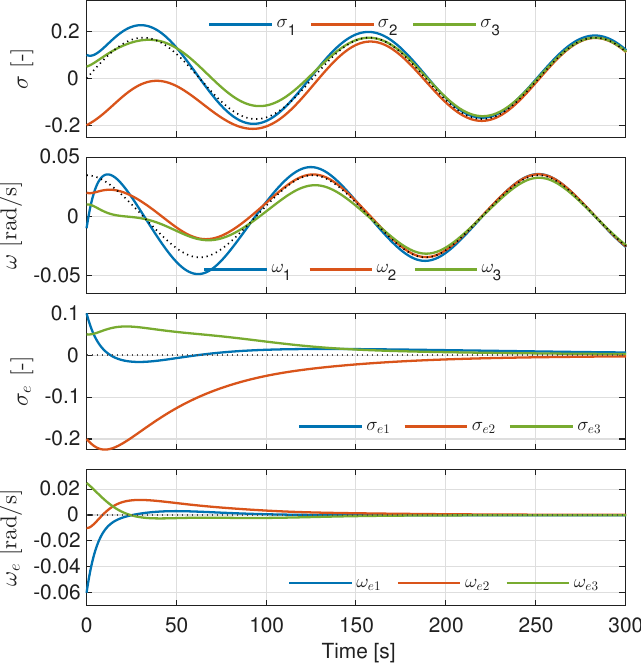}}%
    {\figplaceholder{sim\_tracking: error MRP, $\omega_e$, $h$ vs.\ time; $V_e$ monotone}}
  \caption{Slow sinusoidal tracking ($\bsig_d = 0.3\sin(0.05t)\,\mathbf{1}$, $\bomega_d$ from kinematics, $\bh_0 = \bzero$). Top to bottom: MRPs, angular velocity, MRPs error, angular velocity error.}
  \label{fig:tracking}
\end{figure}

Figure~\ref{fig:safety} stresses the momentum barrier ($H_0 = 0.96\,h_{\max}$).  The proposed controller keeps all wheels strictly inside the limits.  The classical MRP controller, with no momentum management, makes one wheel exceed $h_{\max}$ during the transient. The attitude, angular-velocity and torque histories are otherwise comparable, showing that the embedded momentum protection has little performance cost.

\subsection{Tracking: Sinusoidal Reference}

Figure~\ref{fig:tracking} validates Proposition~\ref{prop:tracking} on the slow sinusoidal reference of the caption, using $k_p = 0.27\,u_{\max}/\sqrt{1+R_{\max}}$ (smaller $k_p$ enlarges $\xi^*$, easing (A3)): $\|\bomega_d\|_{\max}\approx 0.06$~rad/s satisfies (A2) (whose bound is $(h_{\max}-H_0)/\lambda_{\max}(\bJ)\approx 0.32$~rad/s), and the a-priori bound~\eqref{eq:xi_apriori} evaluates to $\|\bxi\|_\infty\leq 0.161$~N\,m~$<\xi^*\approx 0.179$~N\,m, so (A3) holds.  The error state converges to zero, all momenta remain within $h_{\max}$, and $V_e$ decreases monotonically, consistent with Proposition~\ref{prop:tracking}.

\subsection{Comparison with CLF-CBF-QP}

\begin{figure}[!t]
  \centering
  \IfFileExists{figures/sim_qp.pdf}%
    {\includegraphics[width=\columnwidth]{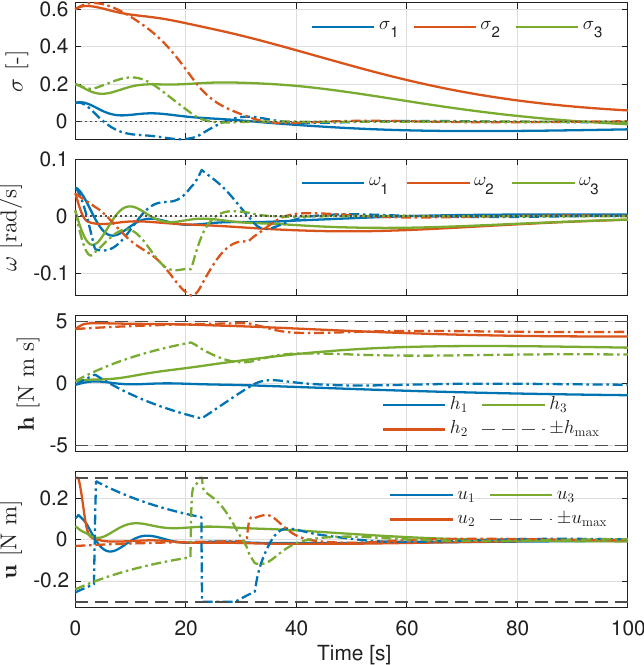}}%
    {\figplaceholder{sim\_qp: $h$ vs.\ time + computation time comparison (proposed vs.\ CLF-CBF-QP)}}
  \caption{Safety stress-test ($H_0=0.96h_{\max}$): proposed controller (solid) vs.\ OD-CLF-CBF-QP~\cite{shahraki2025cbf} (dash-dot). Panels as in Fig.~\ref{fig:safety} (without inset).}
  \label{fig:cbf}
\end{figure}

Figure~\ref{fig:cbf} compares the proposed controller against the OD-CLF-CBF-QP in \cite{shahraki2025cbf} under the same near-saturated conditions, with parameters $\alpha = 0.05$, $\nu = 10$, $p_\rho = 0.1$, $p_\delta = 100$, $\mathbf Q = \mathbf I_6$ and $\mathbf R = \nu \mathbf I_3$, as in \cite{shahraki2025cbf}. A sweep over the gains of both controllers ($k_p\in[0.25,0.95]\cdot u_{\max}/\sqrt{1+R_{\max}}$, $k_d\in[0.05,3]$; $\alpha\in[0.025,0.2]$, $\nu\in\{1,10,100\}$) shows the comparison does not hinge on tuning: every configuration respects both constraints, and all QP tunings settle faster (27--105~s vs.\ 234~s for the best explicit tuning, the nominal above) but with a control effort $\int\|\bu\|^2\mathrm{d}t$ at least $3\times$ larger. This is a structural trade-off stemming from~\eqref{eq:kp_bound}, which caps the explicit law's aggressiveness in exchange for analytic guarantees. OD-CLF-CBF-QP solves a QP at every step, with barriers $h_{\max}\pm h_i$ and a relaxed CLF constraint; it has weaker stability guarantees (convergence only monitored through its slack variable) and cost growing cubically with the number of variables, while the proposed explicit formula has fixed $\mathcal{O}(1)$ cost. Over 1000 simulations with random initial conditions, the average step time is 0.0037~ms for the proposed controller versus 0.73~ms for the QP, a $\sim$200x speedup that enables implementation on compute-limited hardware.

\section{Conclusion}
\label{sec:conclusion}

We have presented an explicit Lyapunov attitude controller that achieves GAS regulation while enforcing hard torque and wheel-momentum constraints, together with a tracking extension under explicit, a-priori checkable feasibility conditions. Simulations confirm constraint satisfaction and a $\sim$200x per-step speedup over an OD-CLF-CBF-QP baseline, at the cost of slower (though guaranteed) convergence.

Future work includes the inclusion of disturbances, other typical wheel restrictions such as avoiding the nonlinearities arising at zero-speed crossings, redundant wheels, and relaxed feasibility conditions for aggressive tracking. The presented law is also useful as a terminal law in Model Predictive Control~\cite{mayne2000mpc,hegre2005mpc}: inside a sublevel set $V\!\leq\!c$ it guarantees recursive feasibility and constraint satisfaction, shortening the horizon.

\section*{Acknowledgment}
\label{sec:acknowledgment}

J.~Urrios acknowledges an FPU grant from MICIU. The authors acknowledge support of grant PID2023-147623OB-I00 funded by MICIU/AEI/10.13039/501100011033 and by ``ERDF A way of making Europe''.

\bibliographystyle{IEEEtran}
\bibliography{refs}

\end{document}